\documentclass{amsart}

\usepackage[T1]{fontenc}	
\usepackage{graphicx}
\usepackage{amsmath, amsthm, amssymb,centernot}
\usepackage{mathrsfs}
\usepackage{pdfsync, comment}

\usepackage[usenames,dvipsnames]{xcolor}
\usepackage[inline,shortlabels]{enumitem}
\usepackage[sort,numbers]{natbib}								
\usepackage[hyphens]{url}									
\usepackage{dsfont}									

\usepackage[
						colorlinks,				
						breaklinks,unicode,
						citecolor=OliveGreen,
						linkcolor=Maroon
						]{hyperref} 

\usepackage{mathtools}
\usepackage{scrextend}						
\usepackage{xparse}

\newcommand{\boldpi}{{\boldsymbol \pi}}

\ExplSyntaxOn
\NewDocumentCommand{\makeabbrev}{mmm}
 {
  \yoruk_makeabbrev:nnn { #1 } { #2 } { #3 }
 }

\cs_new_protected:Npn \yoruk_makeabbrev:nnn #1 #2 #3
 {
  \clist_map_inline:nn { #3 }
   {
    \cs_new_protected:cpn { #2 } { #1 { ##1 } }
   }
 }
 \ExplSyntaxOff

\makeabbrev{\textbf}{tbf#1}{a,b,c,d,e,f,g,h,i,j,k,l,m,n,o,p,q,r,s,t,u,v,w,x,y,z,A,B,C,D,E,F,G,H,I,J,K,L,M,N,O,P,Q,R,S,T,U,V,W,X,Y,Z}

\makeabbrev{\textbf}{bf#1}{a,b,c,d,e,f,g,h,i,j,k,l,m,n,o,p,q,r,s,t,u,v,w,x,y,z,A,B,C,D,E,F,G,H,I,J,K,L,M,N,O,P,Q,R,S,T,U,V,W,X,Y,Z}

\makeabbrev{\textsf}{tsf#1}{a,b,c,d,e,f,g,h,i,j,k,l,m,n,o,p,q,r,s,t,u,v,w,x,y,z,A,B,C,D,E,F,G,H,I,J,K,L,M,N,O,P,Q,R,S,T,U,V,W,X,Y,Z}

\makeabbrev{\mathsf}{mss#1}{a,b,c,d,e,f,g,h,i,j,k,l,m,n,o,p,q,r,s,t,u,v,w,x,y,z,A,B,C,D,E,F,G,H,I,J,K,L,M,N,O,P,Q,R,S,T,U,V,W,X,Y,Z}

\makeabbrev{\mathfrak}{mf#1}{a,b,c,d,e,f,g,h,i,j,k,l,m,n,o,p,q,r,s,t,u,v,w,x,y,z,A,B,C,D,E,F,G,H,I,J,K,L,M,N,O,P,Q,R,S,T,U,V,W,X,Y,Z}

\makeabbrev{\mathrm}{mrm#1}{a,b,c,d,e,f,g,h,i,j,k,l,m,n,o,p,q,r,s,t,u,v,w,x,y,z,A,B,C,D,E,F,G,H,I,J,K,L,M,N,O,P,Q,R,S,T,U,V,W,X,Y,Z}

\makeabbrev{\mathbf}{mbf#1}{a,b,c,d,e,f,g,h,i,j,k,l,m,n,o,p,q,r,s,t,u,v,w,x,y,z,A,B,C,D,E,F,G,H,I,J,K,L,M,N,O,P,Q,R,S,T,U,V,W,X,Y,Z}

\makeabbrev{\mathcal}{mc#1}{A,B,C,D,E,F,G,H,I,J,K,L,M,N,O,P,Q,R,S,T,U,V,W,X,Y,Z}

\makeabbrev{\mathbb}{mbb#1}{A,B,C,D,E,F,G,H,I,J,K,L,M,N,O,P,Q,R,S,T,U,V,W,X,Y,Z}

\makeabbrev{\mathscr}{ms#1}{A,B,C,D,E,F,G,H,I,J,K,L,M,N,O,P,Q,R,S,T,U,V,W,X,Y,Z}

\makeabbrev{\mathrm}{#1}{
Id,id,ran,rk,diag,stab,ann,conv,pr,ev,tr,End,Hom,sgn,im,op,can,fin,ext,red,tot,
rot,usc,lsc,Lip,LocLip,lip,bSymLip,osc,AC,loc,spec,coz,z,
supp,Opt,Adm,Cpl,Geo,OptGeo,GeoAdm,GeoCpl,GeoSel,reg,
bd,co,Ric,Exp,dExp,dist,seg,Seg,cut,fcut,Cut,SDiff,Iso,Isom,diam,cl,Homeo,Diff,Der,vol,dvol,inj,relint, Graph, sub,
var,law,Var,Poi,Gam,pa,so,iso,fs,inv,pqi,mix,
TestF,
}

\makeabbrev{\mathsf}{#1}{DP,CD,QCD,RQCD,BE,MCP,Ent,wMTW,MTW,RCD,EVI,Irr,IH,SC,wFe,UP}

\makeabbrev{\mathsc}{#1}{mmaf,cg,cc}

\DeclareMathOperator{\bLip}{Lip_{\it b}}

\newcommand{\T}{\tau} 
\newcommand{\eps}{\varepsilon}

\renewcommand{\complement}{\mathrm{c}}
\newcommand{\mathsc}[1]{\text{\textsc{#1}}}
\newcommand{\emparg}{{\,\cdot\,}}
\newcommand{\slo}[2][]{\abs{\mathrm{D}#2}_{#1}}
\newcommand{\Ch}[1][]{\mathsf{Ch}_{#1}}

\newcommand{\as}[1]{\quad #1\text{-a.e.}}

\newcommand{\DzLocB}[1]{\mbbL^{#1}_{\loc,b}}

\DeclareMathOperator{\eqdef}{\coloneqq}

\let\epsilon\varepsilon

\newcommand{\rar}{\rightarrow}
\newcommand{\nlim}{\lim_{n}}								

\newcommand{\nliminf}{\liminf_{n }}

\newcommand{\diff}{\mathop{}\!\mathrm{d}}						
\newcommand{\ttabs}[1]{\lvert#1\rvert}	
	
\newcommand{\abs}[1]{\left\lvert#1\right\rvert}						

\newcommand{\norm}[1]{\left\lVert#1\right\rVert}					
\newcommand{\set}[1]{\left\{#1\right\}}							
\newcommand{\tset}[1]{\big\{#1\big\}}							
\newcommand{\ceiling}[1]{\left\lceil#1\right\rceil}					
\newcommand{\tonde}[1]{\left(#1\right)}							
\newcommand{\ttonde}[1]{\big({#1}\big)}
\newcommand{\quadre}[1]{\left[#1\right]}							

\newcommand{\Li}[2][]{\mathrm{L}_{#1}(#2)}						
\newcommand{\rep}[1]{\hat{#1}}								
\newcommand{\scalar}[2]{\left\langle #1 \,\middle |\, #2\right\rangle}		
\newcommand{\sym}[1]{{\scriptscriptstyle{(#1)}}}
\newcommand{\tym}[1]{{\scriptscriptstyle{\times #1}}}

\DeclareSymbolFont{symbolsC}{U}{pxsyc}{m}{n}
\SetSymbolFont{symbolsC}{bold}{U}{pxsyc}{bx}{n}
\DeclareFontSubstitution{U}{pxsyc}{m}{n}
\DeclareMathSymbol{\medcirc}{\mathbin}{symbolsC}{7}
\DeclareSymbolFont{symbolsZ}{OMS}{pxsy}{m}{n}
\SetSymbolFont{symbolsZ}{bold}{OMS}{pxsy}{bx}{n}
\DeclareFontSubstitution{OMS}{pxsy}{m}{n}

\newcommand{\seq}[1]{\tonde{#1}}								

\newcommand{\Cb}{\mcC_b}									
\newcommand{\Cz}{\mcC_0}									

\newcommand{\pfwd}{\sharp}

\DeclareMathOperator*{\essinf}{essinf}
\DeclareMathOperator{\car}{\mathds 1}
\DeclareMathOperator{\emp}{\varnothing} 
\newcommand{\N}{{\mathbb N}}
\newcommand{\R}{{\mathbb R}}

\newcommand{\iref}[1]{\ref{#1}}
\newcommand{\comma}{\,\,\mathrm{,}\;\,}
\newcommand{\semicolon}{\,\,\mathrm{;}\;\,}
\newcommand{\fstop}{\,\,\mathrm{.}}

\newcommand{\blue}[1]{{\color{blue} {#1}}}

\let\temp\phi
\let\phi\varphi
\let\varphi\temp
\newcommand{\hr}[1]{{\bar\mssd}_{#1}} 						

\newcommand{\Rad}[2]{\mathsf{Rad}_{#1,#2}}

\newcommand{\SL}[2]{\mathsf{SL}_{#1,#2}}

\allowdisplaybreaks

\numberwithin{equation}{section}

\theoremstyle{plain}
\newtheorem{thm}{Theorem}[section]
\newtheorem*{thm*}{Theorem}
\newtheorem*{mthm*}{Main Theorem}
\newtheorem{prop}[thm]{Proposition}
\newtheorem*{prop*}{Proposition}
\newtheorem{lem}[thm]{Lemma}
\newtheorem{cor}[thm]{Corollary}

\theoremstyle{definition}
\newtheorem{defs}[thm]{Definition}
\newtheorem*{defs*}{Definition}

\theoremstyle{remark}
\newtheorem{rem}[thm]{Remark}
\newtheorem*{rem*}{Remark}
\newtheorem{ese}[thm]{Example}
\newtheorem*{quest*}{Question}

\renewcommand{\paragraph}[1]{\medskip{\noindent\emph{#1}.\quad}}

\makeatletter   
\def\@fnsymbol#1{\ensuremath{\ifcase#1\or *\or \mathsection \or  \mathparagraph \or \dagger\or
    \ddagger \or \|\or **\or \dagger\dagger
   \or \ddagger\ddagger \else\@ctrerr\fi}}
 \makeatother 

\makeatletter
\newcommand\thankssymb[1]{\textsuperscript{\@fnsymbol{#1}}}
\makeatother

\makeatother

\makeatletter
\def\l@subsection{\@tocline{2}{0pt}{2.5pc}{5pc}{}}
\makeatother

\begin{document}
\title[Properties of QCD Spaces]{Sobolev-to-Lipschitz Property on $\QCD$-spaces\\and Applications\thankssymb{4}}

\author[L.~Dello Schiavo]{Lorenzo Dello Schiavo\thankssymb{1}}
\address{IST Austria, Am Campus 1, 3400 Klosterneuburg, Austria}
\email{lorenzo.delloschiavo@ist.ac.at}
\thanks{\thankssymb{1}IST Austria, e-mail: lorenzo.delloschiavo@ist.ac.at. Orcid: 0000-0002-9881-6870}

\author[K.~Suzuki]{Kohei Suzuki\thankssymb{2}}
\address{Fakult\"at f\"ur Mathematik, Universit\"at Bielefeld, D-33501, Bielefeld, Germany}
\email{ksuzuki@math.uni-bielefeld.de}
\thanks{\thankssymb{2}Institut f\"ur Mathematik, Universit\"at Bielefeld, e-mail: ksuzuki@math.uni-bielefeld.de. Orcid: 0000-0002-3048-9738}

\thanks{\thankssymb{4}The authors are grateful to Dr.~Bang-Xian Han for helpful discussions on the Sobolev-to-Lipschitz property on metric measure spaces, and to Professor~Kazuhiro Kuwae, Professor Emanuel Milman, Dr.~Giorgio Stefani, and Dr.~Gioacchino Antonelli for reading a preliminary version of this work and for their valuable comments and suggestions.
Finally, they wish to express their gratitude to two anonymous Reviewers whose suggestions improved the presentation of this work.
\\
\indent L.D.S.\ gratefully acknowledges funding of his position by the Austrian Science Fund (FWF) grant F65, and by the European Research Council (ERC, grant No.~716117, awarded to Prof.\ Dr.~Jan Maas).
\\
\indent K.S.\ gratefully acknowledges funding by: the JSPS Overseas Research Fellowships, Grant Nr. 290142; World Premier International Research Center Initiative (WPI), MEXT, Japan; JSPS Grant-in-Aid for Scientific Research on Innovative Areas ``Discrete Geometric Analysis for Materials Design'', Grant Number 17H06465; and the Alexander von Humboldt Stiftung, Humboldt-Forschungsstipendium.}

\begin{abstract}
We prove the Sobolev-to-Lipschitz property for metric measure spaces satisfying the quasi curvature-dimension condition recently introduced in E.~Milman, \emph{The Quasi Curvature-Dim\-ension Condition with applications to sub-Riemannian manifolds}, Comm.\ Pure Appl.~Math.\ (to appear, arXiv:1908.\ 01513v5). We provide several applications to properties of the corresponding heat semigroup.
In particular, under the additional assumption of infinitesimal Hilbertianity, we show the Varadhan short-time asymptotics for the heat semigroup with respect to the distance, and prove the irreducibility of the heat semigroup.
These result apply in particular to large classes of (ideal) sub-Riemannian manifolds.
\end{abstract}


\maketitle


\date{\today}
\maketitle

\section{Introduction}
In~\cite{Mil21}, E.~Milman introduced the notion of \emph{quasi curvature-dimension condition}~$\QCD$ for a metric measure space~$(X,\mssd,\mssm)$, simultaneously generalizing Lott--Villani--Sturm's \emph{curvature-dimension condition}~$\CD(K,N)$ with finite~$N$, \cite{LotVil09,Stu06a, Stu06b} and the \emph{measure contraction property}~$\MCP$,~\cite{Stu06b, Oht07b}.
As discussed in~\cite{Mil21}, the class of~$\QCD$ spaces notably includes large families of (ideal) sub-Riemann manifolds, thus aiming to provide a unified perspective of (non-smooth) Riemannian, Finsler, and sub-Riemann\-ian geometry.

In this note, we collect some metric-measure properties of a metric measure space~$(X,\mssd,\mssm)$ satisfying the~$\QCD$ condition.
As a main result, Theorem~\ref{t:Main}, we show the Sobolev-to-Lipschitz property, see~\ref{eq:SL} below.

In light of recent developments in metric analysis, the property~\ref{eq:SL} has turned out to be significant in relating differentiable and metric measure structures.
For instance,
\begin{itemize}
\item under~\ref{eq:SL}, the {Bakry--\'Emery (synthetic Ricci) curvature (lower) bound} $\BE$ is equivalent to the {Riemannian curvature-dimension condition} $\RCD$, Ambrosio--Gigli--Savar\'e \cite{AmbGigSav15}. 
The statement is sharp, in the sense that~$\BE$ without~\ref{eq:SL} does not imply~$\RCD$, Honda~\cite{Hon18};

\item together with~\ref{eq:SL}, the~$\BE$ condition implies the $L^\infty$-to-Lipschitz regularization of the heat semigroup, Ambrosio--Gigli--Savar\'e \cite{AmbGigSav14} (in the sub-Riemannian setting see Stefani~\cite{Ste20});

\item together with a Rademacher-type property for~$(X,\mssd,\mssm)$, see~\ref{eq:Rad} below,~\ref{eq:SL} implies the coincidence of the intrinsic distance and the given distance~$\mssd$, and also implies the integral Varadhan short-time asymptotic for the heat semigroup in a variety of settings (see~\cite[Thm.~4.25]{LzDSSuz20}), and furthermore, for the space of configurations (i.e., locally finite integer-valued point measures) over~$X$, see~\cite[Thm.~6.10]{LzDSSuz21}.
\end{itemize}

Apart from~\ref{eq:SL}, $\QCD$ spaces satisfy the local volume doubling, the Rademacher-type property~\ref{eq:Rad}, and the local versions \ref{eq:RadLoc} and~\ref{eq:SLLoc} of~\ref{eq:Rad} and~\ref{eq:SL} after~\cite{LzDSSuz20}, see~\S\ref{sec: RQCD}. 
When~$(X,\mssd,\mssm)$ is additionally infinitesimally Hilbertian, as an application of the Sobolev-to-Lipschitz property, we obtain:
\begin{itemize}
\item the coincidence of the distance~$\mssd$ with the intrinsic distance~$\mssd_{\Ch}$ of the Cheeger energy~$\Ch$ of~$(X,\mssd,\mssm)$, Theorem~\ref{t:Distances};
\item the integral Varadhan short-time asymptotic for the heat semigroup with respect to the Hausdorff distance induced by~$\mssd$, Theorem~\ref{t:Varadhan};
\item in the compact case, and assuming as well the measure contraction property~$\MCP$, the pointwise Varadhan short-time asymptotic for the heat semigroup with respect to the Hausdorff distance induced by~$\mssd$, Corollary~\ref{c:pVaradhan};
\item the irreducibility of the Dirichlet form~$\Ch$, Corollary~\ref{c:Irr}.
\end{itemize}
For these results, we make full use of fundamental relations between Dirichlet forms and metric measure spaces developed in~\cite{LzDSSuz20}. 

Regarding the irreducibility, we note that the same proof of Corollary~\ref{c:Irr} applies as well to~$\RCD(K,\infty)$ spaces with infinite volume measure, which seems not explicitly proved in the existing literature; see Remark~\ref{rem: IRCD} for a more detailed discussion. 

Our results on~$\QCD$ spaces may be specialized to ideal sub-Riemannian manifolds satisfying the quasi curvature-dimension condition, such as: (ideal generalized $H$-type) Carnot groups, Heisenberg groups, corank-$1$ Carnot groups, the Grushin plane, and several $H$-type foliations, Sasakian and $3$-Sasakian manifolds.

\section*{Declarations}
The authors have no conflicts of interest to declare that are relevant to the content of this article.

\section{Milman's Quasi Curvature-Dimension-Condition}
By a \emph{metric measure space}~$(X,\mssd,\mssm)$ we shall always mean a complete and separable metric space~$(X,\mssd)$, endowed with a Borel measure~$\mssm$ finite on $\mssd$-bounded sets and with full topological support.
In order to rule out trivial cases, we assume that~$\mssm$ is atomless, which makes~$X$ uncountable.
We say that~$(X,\mssd)$ is \emph{proper} if all closed balls are compact.
We let~$\mcC_c(X)$, resp.~$\Cz(X)$, $\Cb(X)$, be the space of continuous compactly supported, resp.\ continuous vanishing at infinity, continuous bounded, functions on~$X$.

We denote by~$\msP(X)$, resp.~$\msP_c(X)$,~$\msP^\mssm(X)$, the space of all Borel probability measures on~$(X,\mssd)$, resp.\ (additionally) compactly supported, (additionally) absolutely continuous w.r.t.~$\mssm$, and by
\begin{align*}
\msP_2(X)\eqdef \set{\mu\in\msP(X): \int_X \mssd(x,x_0)^2\diff\mu <\infty}
\end{align*}
the \emph{$L^2$-Wasserstein space} over~$(X,\mssd)$, endowed with the \emph{$L^2$-Wasserstein distance}
\begin{align}\label{eq:Wasserstein}
W_2(\mu_0,\mu_1)\eqdef \quadre{\inf_\pi\int_{X^\tym{2}} \mssd(x,y)^2\diff\pi(x,y)}^{1/2}\comma
\end{align}
the infimum running over all couplings~$\pi\in\msP(X^\tym{2})$ of~$(\mu_0,\mu_1)$.
We denote by~$\Opt(\mu_0,\mu_1)$ the set of minimizers in~\eqref{eq:Wasserstein}, always non-empty.

Set~$I\eqdef [0,1]$.
We write~$\Geo(X,\mssd)$ for the space of all constant-speed geodesics in~$(X,\mssd)$ parametrized on~$I$, itself a complete separable metric space when endowed with the supremum distance~$\mssd_\infty$ induced by~$\mssd$.
By Lisini's \emph{superposition principle}~\cite[Thm.~4]{Lis07} (cf.~\cite[Thm.~2.10]{AmbGig11}), every $W_2$-absolutely continuous curve~$\seq{\mu_t}_{t\in I}$ may be lifted to a dynamical plan~$\boldpi\in\msP(\mcC(I;X))$ satisfying~$(\ev_t)_\pfwd\boldpi=\mu_t$ for every~$t\in I$, where~$\ev_t\colon \gamma\mapsto \gamma_t$ is the evaluation map at time~$t$.
Furthermore, a curve~$\seq{\mu_t}_t$ is a $W_2$-geodesic if and only if~$\boldpi$ is concentrated on~$\Geo(X,\mssd)$ and
\begin{align}\label{eq:WassersteinGeodesic}
W_2(\mu_0,\mu_1)^2=\int_{\mcC(I;X)} \int_I \abs{\dot\gamma}_t^2\diff t \diff\boldpi(\gamma)\comma
\end{align}
in which case we say that~$\boldpi$ is an \emph{optimal dynamical plan connecting~$\mu_0$ and~$\mu_1$}. We write~$\OptGeo(\mu_0,\mu_1)$ for the set of all such plans.

Whenever~$(Y,\T)$ is a Polish space, the narrow topology~$\T_\mrmn$ on the space~$\msP(Y)$ of Borel probability measures on~$Y$ is defined as the topology induced by duality with continuous bounded functions on~$Y$.
Since~$(Y,\T)$ is Polish, $(\msP(Y),\T_\mrmn)$ is Polish as well, $\T_\mrmn$ is characterized by the convergence of sequences, and a sequence~$\seq{\mu_n}_n$ converges narrowly if and only if
\begin{align*}
\int_X f\diff\mu_n \xrightarrow{\ n\to \infty \ } \int_X f \diff\mu\comma \qquad f\in\Cb(X) \fstop
\end{align*}

We collect here for further reference the following standard fact.
\begin{prop}[Stability of dynamical optimality]\label{p:DynStability}
For~$i=0,1$ let~$\seq{\mu^n_i}_n\subset \msP_2(X)$ and fix~$\boldpi^n\in \OptGeo(\mu^n_0,\mu^n_1)$.
If~$\boldpi^n$ narrowly converges to~$\boldpi\in\msP(\mcC(I;X))$, then~$\boldpi$ is concentrated on~$\Geo(X,\mssd)$, and~$\boldpi\in\OptGeo\ttonde{(\ev_0)_\pfwd \boldpi, (\ev_1)_\pfwd \boldpi}$.
\begin{proof}
Consequence of the stability of $W_2$-optimality~\cite[Prop.~2.5]{AmbGig11}, the continuity of~$(\ev_0,\ev_1)\colon \Geo(X,\mssd)\to X^\tym{2}$, and the Continuous Mapping Theorem.
\end{proof}
\end{prop}

\begin{defs}[Monge space,~{\cite[Dfn.~3.1]{Mil21}}]\label{d:Monge}
A metric measure space~$(X,\mssd,\mssm)$ is a \emph{Monge space} if for every~$\mu_0,\mu_1\in\msP_2(X)$ with~$\mu_0\ll \mssm$ the following holds:
\begin{enumerate}[$(a)$]
\item\label{i:d:Monge:1} there exists a unique optimal dynamical plan~$\boldpi\in \OptGeo(\mu_0,\mu_1)$, hence $\Opt(\mu_0,\mu_1)$ consists of the unique optimal plan~$(\ev_0,\ev_1)_\pfwd \boldpi$;
\item\label{i:d:Monge:2} $(X,\mssd,\mssm)$ has \emph{good transport behavior} (cf.~\cite[Dfn.~3.1]{Kel17}), i.e.~$\boldpi$ is induced by a map, viz.~$\boldpi=\mbfT_\pfwd\mu_0$ for some~$\mbfT\colon X\to \Geo(X,\mssd)$;
\item\label{i:d:Monge:3} $(X,\mssd,\mssm)$ has the \emph{strong interpolation property}~\cite[p.~523]{Kel17}, i.e.\ the optimal dynamical plan~$\boldpi$ in~\ref{i:d:Monge:2} satisfies~$\mu_t\eqdef (\ev_t)_\pfwd\boldpi\ll \mssm$ for all~$t\in [0,1)$.
\end{enumerate}

We write~$\rho_t$ for the Radon--Nikod\'ym density of~$\mu_t$ w.r.t.~$\mssm$.
\end{defs}

Let us now collect some properties of Monge spaces.

\begin{rem}
Every geodesic Monge space is ($2$-)\emph{essentially non-branching}~\cite[Dfn.~2.10]{Kel17} by~\cite[Prop.~3.6]{Kel17}.
\end{rem}

\begin{cor}\label{c:ContPi}
Let~$(X,\mssd,\mssm)$ be a geodesic Monge space and fix~$K_i\Subset X$, $i=0,1$. Then, the map
\begin{equation}\label{eq:Pi}
\Pi\colon (\mu_0,\mu_1)\longmapsto \boldpi\in\OptGeo(\mu_0,\mu_1)
\end{equation}
is a continuous map~$\msP^\mssm(K_0)\times\msP(K_1)\to \msP(\Geo(X,\mssd))$ when~$\msP^\mssm(K_0)\times\msP(K_1)$ is endowed with the product of the narrow topologies, and $ \msP(\Geo(X,\mssd))$ is endowed with the narrow topology.
\begin{proof}
Firstly, note that~$\Pi$ is well-defined by Definition~\ref{d:Monge}\ref{i:d:Monge:1}.
Secondly, recall that the space of geodesics
\begin{equation*}
\Geo(K_0,K_1)\eqdef \set{\gamma\in \Geo(X,\mssd): \ev_i \gamma\in K_i}
\end{equation*}
is a compact subset of~$\Geo(X,\mssd)$.
As a consequence,~$\msP(\Geo(K_0,K_1))$ is narrowly compact metrizable, and it suffices to show the continuity of~$\Pi$ along sequences.
To this end let~$\seq{\mu^n_0}_n\subset \msP^\mssm(K_0)$, and~$\seq{\mu^n_1}_n\subset\msP(K_1)$, be narrowly convergent to~$\mu_0\in\msP^\mssm(K_0)$, resp.~$\mu_1\in\msP(K_1)$.
Since~$\seq{\Pi(\mu^n_0,\mu^n_1)}_n\subset \msP(\Geo(K_0,K_1))$, it admits a non-relabeled narrowly convergent subsequence.
Let~$\boldpi$ be its limit, and note that~$\boldpi\in\msP(\Geo(X,\mssd))$ is an optimal dynamical plan by Proposition~\ref{p:DynStability}.
By continuity of~$(\ev_0,\ev_1)\colon \Geo(X,\mssd)\to X^\tym{2}$, the narrow convergence of~$\boldpi^n$ to~$\boldpi$, and the Continuous Mapping Theorem, we conclude that~$\boldpi\in \OptGeo(\mu_0,\mu_1)$.
Since the latter is a singleton by Definition~\ref{d:Monge}\ref{i:d:Monge:1}, we have therefore~$\boldpi=\Pi(\mu_0,\mu_1)$.
Since the subsequence was arbitrary, we have concluded that~$\nlim \Pi(\mu^n_0,\mu^n_1)=\nlim \boldpi^n=\boldpi=\Pi(\mu_0,\mu_1)$, which proves the assertion.
\end{proof}
\end{cor}

The following is a consequence of the sole \emph{interpolation property}~\cite[Dfn.~4.2]{Kel17}.

\begin{cor}\label{c:ContinuitySet}
Let~$(X,\mssd,\mssm)$ be a geodesic Monge space. Then, every ball~$B\subset X$ is a continuity set for~$\mssm$, i.e.~$\mssm\, \partial B=0$.
In particular, the sphere
\[
S_r(x)\eqdef\set{y\in X: \mssd(x,y)=r}
\]
is $\mssm$-negligible for every~$x\in X$ and every~$r>0$.
\begin{proof}
Since~$(X,\mssd,\mssm)$ has the strong interpolation property (Dfn.~\ref{d:Monge}\iref{i:d:Monge:3}), it has in particular the interpolation property, and is therefore \emph{strongly non-degenerate}~\cite[Dfn.~4.4]{Kel17} by~\cite[Lem.~4.5]{Kel17}.
In particular, it is \emph{non-degenerate}, i.e., for every Borel~$A\subset X$ with~$\mssm A>0$ and every~$x\in X$ it holds that~$\mssm A_{t,x}>0$ for every~$t\in (0,1)$, where
\begin{align*}
A_{t,x}\eqdef \set{\gamma_t : \gamma\in \Geo(X,\mssd)\comma \gamma_0\in A, \gamma_1=x} \fstop
\end{align*}

Now, argue by contradiction that there exist~$x_0\in X$ and~$r_0>0$ with $\mssm S_{r_0}(x_0)>0$.
On the one hand, since~$\mssm$ is $\sigma$-finite we can find~$r\in (0,r_0)$ with~$\mssm S_r(x_0)=0$.
On the other hand, since~$S_{r}(x_0)=\ttonde{S_{r_0}(x_0)}_{t,x_0}$ for~$t\eqdef r/r_0\in (0,1)$, the non-degeneracy implies~$\mssm S_{r}(x_0)>0$, a contradiction.

Thus, every sphere~$S_r(x)\subset X$ is $\mssm$-negligible.
Since~$(X,\mssd)$ is geodesic,~$S_r(x)=\partial B_r(x)$, and the first assertion follows.
\end{proof}
\end{cor}

The following generalization of Lott--Sturm--Villani curvature-dimension condition was recently introduced by E.~Milman,~\cite{Mil21}.
For~$K\in\R$,~$N\in (1,\infty)$ and~$t\in (0,1)$, denote by~$\tau^\sym{t}_{K,N}$ the \emph{dynamical distortion coefficient} of the model space of constant sectional curvature~$\frac{K}{N-1}$ and dimension~$\ceiling{N}$, e.g.~\cite[Eqn.~(2.2)]{Mil21}.

\begin{defs}[Quasi Curvature-Dimension Condition, {\cite[Dfn.s~2.3,2.8]{Mil21}}]
For $Q\geq 1$, $K\in\R$, and~$N\in (1,\infty)$, a geodesic Monge space~$(X,\mssd,\mssm)$ satisfies the \emph{quasi curvature-dimension condition}~$\QCD(Q,K,N)$ if, for every~$\mu_0,\mu_1\in\msP^\mssm_c(X)$,
\begin{align}\label{eq:QCD}
\rho_t^{-1/N}(\gamma_t) \geq Q^{-1/N} \tonde{\tau^\sym{1-t}_{K,N}\ttonde{\mssd(\gamma_0,\gamma_1)}\, \rho_0^{-1/N}(\gamma_0)+ \tau^\sym{t}_{K,N}\ttonde{\mssd(\gamma_0,\gamma_1)}\, \rho_1^{-1/N}(\gamma_1)}
\end{align}
for every~$t\in (0,1)$ and $\boldpi$-a.e.~$\gamma\in\Geo(X,\mssd)$.

We further say that~$(X,\mssd,\mssm)$ satisfies the \emph{regular quasi curvature-dimension condition} $\QCD_\reg(Q,K,N)$ if it satisfies~$\QCD(Q,K,N)$ for some~$Q,K,N$ as above and additionally the \emph{measure contraction property}~$\MCP(K',N')$ for some~$K'\in \R$ and~$N'\in (1,\infty)$.
\end{defs}

In the following, we omit the indices~$Q$,~$K$, and~$N$ whenever not relevant.
We refer to~\cite{Mil21} for a thorough discussion of examples of spaces satisfying the~$\QCD$ condition.
We stress that they include all~$\CD$ spaces (for the choice~$Q=1$, see~\cite{Mil21}), and various classes of sub-Riemannian manifolds satisfying~$\MCP$ (see~\cite[Prop.~2.4]{Mil21}).

Since the right-hand side of~\eqref{eq:QCD} depends on~$Q$ only by its linear dependence on the constant~$Q^{-1/N}$, the proof of the following result is readily adapted from the one of the analogous assertion under the curvature-dimension condition~$\CD$ in~\cite[Thm.~2.3]{Stu06b}.

For fixed~$x_0\in X$ and for every~$r>0$, set~$\overline{B}_r(x_0)\eqdef \set{x\in X: \mssd(x,x_0)\leq r}$, and
\begin{align*} 
v(r)\eqdef \mssm B_r(x_0) \qquad \text{and} \qquad s(r)\eqdef \limsup_{\delta\downarrow 0} \frac{1}{\delta} \, \mssm \ttonde{\overline{B}_{r+\delta}(x_0)\setminus B_r(x_0)} \fstop
\end{align*}
As customary, further define the model volume coefficient
\begin{align*}
s_{K,N}(r)\eqdef \begin{cases} \sin\tonde{\sqrt{\tfrac{K}{N-1}} \, r} & \text{if } K>0\\ r & \text{if } K=0\\ \sinh\tonde{\sqrt{\tfrac{-K}{N-1}} \, r} & \text{if } K<0
\end{cases}\fstop
\end{align*}

\begin{lem}[Generalized Bishop--Gromov inequality]\label{l:LocalDoubling}
Let~$(X,\mssd,\mssm)$ be a metric measure space satisfying~$\QCD(Q,K,N)$ for some~$Q\geq 1$,~$K\in\R$, and~$N\in (1,\infty)$. Then, for every~$0<r\leq R$ \emph{(}with~$R\leq \pi/\sqrt{K/(N-1)}$ if~$K>0$\emph{)},
\begin{align}\label{eq:Sturm:0a}
\frac{s(r)}{s(R)}\geq&\ Q^{-N}\, \tonde{\frac{s_{K,N}(r)}{s_{K,N}(R)}}^{N-1}\comma
\\
\label{eq:Sturm:0b}
\frac{v(r)}{v(R)} \geq&\ Q^{-N}\, \frac{\int_0^r s_{K,N}(t)^{N-1}\diff t}{\int_0^R s_{K,N}(t)^{N-1}\diff t} \fstop
\end{align}
\end{lem}

\begin{rem*}
If~$Q>1$, then~\eqref{eq:Sturm:0b} is not sufficient to conclude that~$v(r)/v(R)\to 1$ as~$r\to R$.
In particular, this implies that the assertion of Corollary~\ref{c:ContinuitySet} does not follow from~\eqref{eq:Sturm:0b} in the obvious way, which makes the Corollary non-void.
\end{rem*}

\begin{proof}[Proof of Lemma~\ref{l:LocalDoubling}]
Let~$A_0,A_1$ be Borel subsets of~$X$ with~$\mssm A_0, \mssm A_1>0$, and set
\begin{align*}
A_{t}\eqdef \set{\gamma_t : \gamma\in\Geo(X,\mssd)\comma \gamma_i\in A_i\comma i=0,1} \fstop
\end{align*}
Following \emph{verbatim} the proof of~\cite[Prop.~2.1]{Stu06b} yields the following $Q$-weighted version of the generalized Brunn--Minkowski inequality:
\begin{align}\label{eq:BrunnMinkowski}
(\mssm A_t)^{1/N}\geq Q^{-1} \tonde{\tau^{\sym{1-t}}_{K,N} (\Theta)\, (\mssm A_0)^{1/N}+ \tau^{\sym{t}}_{K,N} (\Theta)\, (\mssm A_1)^{1/N}}\comma
\end{align}
where
\begin{align*}
\Theta\eqdef \begin{cases} \inf_{x_i\in A_i} \mssd(x_0,x_1) & \text{if } K\geq 0\\ \sup_{x_i\in A_i}\mssd(x_0,x_1) &\text{if } K<0\end{cases} \fstop
\end{align*}

We apply~\eqref{eq:BrunnMinkowski} to~$A_0\eqdef B_\eps(x_0)$ and~$A_1\eqdef \overline{B}_{R+\delta R}(x_0)\setminus B_R(x_0)$ for some~$\eps,\delta>0$, fixed~$0<r<R$, and with~$t\eqdef r/R$.
It is readily seen that
\begin{align*}
A_t\subset \overline{B}_{r+\delta r+\eps r/R}(x_0)\setminus B_{r-\eps r/R}(x_0) \qquad \text{and} \qquad R-\eps\leq \Theta \leq R+\delta R+\eps \fstop
\end{align*}
Thus, by~\eqref{eq:BrunnMinkowski},
\begin{align*}
\mssm&\ttonde{\overline{B}_{r+\delta r+\eps r/R}(x_0)\setminus B_{r-\eps r/R}(x_0)}^{1/N}\geq
\\
&\qquad Q^{-1} \tau^\sym{1-r/R}_{K,N} (R\mp \delta R \mp \eps) \ttonde{\mssm B_\eps(x_0)}^{1/N}
\\
&\qquad\quad+Q^{-1}\tau^{\sym{r/R}}_{K,N}(R\mp \delta R\mp \eps) \, \mssm\ttonde{\overline{B}_{R+\delta R}(x_0)\setminus B_R(x_0)}^{1/N}
\end{align*}
where~$\mp$ is chosen to coincide with~$\sgn(K)$. Letting~$\eps \to 0$,
\begin{align*}
\mssm&\ttonde{\overline{B}_{(1+\delta)r}(x_0)\setminus B_r(x_0)}^{1/N}\geq Q^{-1}\tau^{\sym{r/R}}_{K,N}\ttonde{(1\mp \delta) R}\,  \mssm\ttonde{\overline{B}_{R+\delta R}(x_0)\setminus B_R(x_0)}^{1/N} \fstop
\end{align*}

Since~$\mssm$ does not charge spheres by Corollary~\ref{c:ContinuitySet}, we may rewrite the above inequality as
\begin{align*}
v\ttonde{(1+\delta)r}-v(r) \geq Q^{-N}\, \tau^{\sym{r/R}}_{K,N}\ttonde{(1\mp\delta)R}^N \tonde{v\ttonde{(1+\delta)R}-v(R)} \comma
\end{align*}
hence, making explicit the definition of the distortion coefficients,
\begin{align}\label{eq:Sturm:1}
\frac{v\ttonde{(1+\delta)r}-v(r)}{\delta r}\geq Q^{-N} \frac{v\ttonde{(1+\delta)R}-v(R)}{\delta R} \tonde{\frac{s_{K,N}\ttonde{(1\mp \delta)r}}{s_{K,N}\ttonde{(1\mp \delta)R}}}^{N-1} \fstop
\end{align}
Letting~$\delta\to 0$ in~\eqref{eq:Sturm:1} proves~\eqref{eq:Sturm:0a}.
The inequality~\eqref{eq:Sturm:0b} now follows from~\eqref{eq:Sturm:0a} exactly as in the proof of~\cite[Thm.~2.3]{Stu06b}.
\end{proof}

\blue{Recall that a metric measure space~$(X,\mssd,\mssm)$ is \emph{locally doubling} if for every~$x\in X$ there exists an open set~$U\ni x$ and constants~$C,R>0$ so that
\[
\mssm B_{2r}(y)\leq C \mssm B_r(y)\comma \qquad r\in (0,R)\comma \qquad y\in U\fstop
\]
Following~\cite[Dfn.~3.18]{GigHan18}, we further say that~$(X,\mssd,\mssm)$ is \emph{a.e.-locally doubling} if there exists an $\mssm$-negliglible set~$N\subset X$ so that~$X\setminus N$ is locally doubling when endowed with the restriction of~$\mssd$ and~$\mssm$.
}

As standard corollaries of Lemma~\ref{l:LocalDoubling}, we further have the following properties:

\begin{cor}\label{c:LocalDoubling}
Every $\QCD$ space is locally doubling.
\end{cor}

\begin{cor}\label{c:Proper}
Every $\QCD$ space is proper.
\end{cor}

\section{The Sobolev-to-Lipschitz property on $\QCD$-spaces}
Let~$(X,\mssd,\mssm)$ be a metric measure space.
We denote by~$\Li{f}$ the global Lipschitz constant of a Lipschitz function~$f\colon X\to\R$, and by~$\Lip(\mssd)$, resp.~$\Lip_{bs}(\mssd)$, the space of all Lipschitz functions, resp.\ (additionally) with bounded support, on~$(X,\mssd)$.
We briefly recall the definition of Cheeger energy of a metric measure space. For a function~$f\in \Lip(\mssd)$, define the slope of~$f$ at~$x$ by
\begin{align*}
\slo{f}(x)\eqdef \limsup_{y\rar x} \frac{\abs{f(y)-f(x)}}{\mssd(x,y)}\comma
\end{align*}
where, conventionally,~$\slo{f}(x)=0$ if $x$ is isolated.
The \emph{Cheeger energy}~\cite[Eqn.\ (4.11)]{AmbGigSav14} on~$(X,\mssd,\mssm)$ is the functional
\begin{align*}
\Ch[\mssd,\mssm](f)\eqdef \inf\set{\nliminf \int_{X} \slo{f_n}^2 \diff \mssm: f_n\in \Lip_{bs}(\mssd)\comma L^2(\mssm)\text{-}\nlim f_n=f} \comma
\end{align*}
where, conventionally, $\inf\emp\eqdef+\infty$.
We denote the domain of $\Ch[\mssd,\mssm]$ by
\begin{align*}
W^{1,2}= W^{1,2}(X, \mssd. \mssm)\eqdef \set{f \in L^2(X, \mssm): \Ch[\mssd,\mssm](f)<\infty}\fstop
\end{align*}
We recall that a metric measure space~$(X,\mssd,\mssm)$ is called \emph{infinitesimally Hilbertian} if~$\Ch[\mssd,\mssm]$ is quadratic.
Introduced by N.~Gigli in~\cite[Dfn.~4.19]{Gig12a}, this notion has ever since proven to be a key tool in the study of non-smooth metric measure spaces.

When~$(X,\mssd,\mssm)$ is infinitesimally Hilbertian,~$\Ch[\mssd,\mssm]$ is a strongly local Dirichlet form having square-field operator~$\slo[w]{f}^2$, where $\slo[w]{f}$ is called the weak minimal upper gradient, satisfying ---~by construction~--- the \emph{Rademacher-type property}:
\begin{equation}\tag*{$(\mathsf{Rad})$}\label{eq:Rad}
\Lip_{bs}(\mssd) \subset W^{1,2}\comma \qquad \slo[w]{f} \le \Li{f} \as{\mssm}\comma
\end{equation}
where~$\Lip_{bs}$ denotes the space of Lipschitz functions with bounded support.

The following property has been considered in a variety of non-smooth settings, including e.g.\ configuration spaces~\cite{RoeSch99}, or general metric measure spaces~\cite[Dfn.~4.9]{Gig13}.

\begin{defs}[Sobolev-to-Lipschitz property]
We say that a metric measure space~$(X,\mssd,\mssm)$ satisfies the \emph{Sobolev-to-Lipschitz property} if
\begin{equation}\tag*{$(\mathsf{SL})$}\label{eq:SL}
\begin{gathered}
\text{each~$f\in W^{1,2}\cap L^\infty(\mssm)$ with~$\slo[w]{f}\leq 1$}
\\
\text{has a Lipschitz $\mssm$-modification~$\rep f$ with~$\Li{\rep f} \leq 1$}\fstop
\end{gathered}
\end{equation}
\end{defs}

\blue{
\begin{rem}
The Sobolev-to-Lipschitz property is more commonly phrased without the requirement that~$f\in L^\infty(\mssm)$.
In fact, this is equivalent to~\ref{eq:SL}.
\begin{proof}
Let~$f\in W^{1,2}$ with~$\slo[w]{f}\leq 1$, and set~$f_r\eqdef -r \vee f \wedge r$ for each $r>0$.
By locality of~$\slo[w]{\emparg}$ we have that~$\slo[w]{f_r}\leq 1$ $\mssm$-a.e.\ for every~$r$, hence~$f_r$ has a Lipschitz $\mssm$-modification~$\rep f_r$ with~$\Li{\rep f}\leq 1$, by~\ref{eq:SL}.
Since~$\rep f_r$ is continuous and $\mssm$ has full support, we have that~$\rep f_r\equiv \rep f_s$ everywhere on~$\tset{x\in X:\ttabs{\rep f_s(x)}\leq r}$ for every~$s\geq r$.
We conclude letting~$r\to\infty$.
\end{proof}
\end{rem}

\begin{rem}[On the Rademacher-type and Sobolev-to-Lipschitz properties]
As anticipated above, the Rademacher-type property for~$\Ch[\mssd,\mssm]$ holds \emph{by construction} on \emph{every} metric measure space (even without the assumption of infinitesimal Hilbertianity).
However the Rademacher-type property becomes non-trivial in the general setting of strongly local Dirichlet forms, i.e.\ when~$\Ch[\mssd,\mssm]$ is replaced by any such Dirichlet form~$\mcE$ on~$L^2(\mssm)$.
In this case, the Sobolev-to-Lipschitz property and the Rademacher-type property may be regarded as converse to each other.
For many comments and examples on both properties in this setting, see e.g.~\cite[\S\S3-4]{LzDSSuz20}.

Concerning our terminology, we ought to stress that the Rademacher-type property~\ref{eq:Rad} we defined here does not entail any (strong) differentiability nor any G\^{a}teaux (i.e.\ directional) differentiability of the function involved.
Indeed, even phrasing any such concept of (G\^{a}teaux) differentiability on general metric measure spaces would be highly non-trivial.
While immediate on metric measure spaces (as discussed above), the property~\ref{eq:Rad} is highly non-trivial on general Dirichlet spaces when~$\slo{\emparg}$ is replaced by the square field operator (or even by the energy measure) of a strongly local Dirichlet form, see e.g.~\cite{Kuw96,Stu94,LzDSSuz20}.
In this case, proofs of strong notions of differentiability of Lipschitz functions are available in specific smooth and non-smooth settings, which invariably rely on a combination of G\^{a}teaux differentiability along `sufficiently many' directions together with the uniform bound on such derivatives, provided by~\ref{eq:Rad}; see e.g.~\cite{NekZaj88} for Euclidean spaces,~\cite{BogMay96,EncStr93} for Wiener and Banach spaces,~\cite{RoeSch99} for configuration spaces,~\cite{LzDS19b} for Wasserstein spaces, etc..
\end{rem}
}


In order to discuss the Sobolev-to-Lipschitz property on~$\QCD$ spaces, we recall the following definition by N.~Gigli and~B.-X.~Han,~\cite{GigHan18}.
Firstly, recall from~\cite[Dfn.~5.1]{AmbGigSav14} that a dynamical plan~$\boldpi\in\msP(\mcC(I;X))$ is a \emph{test plan} if
\begin{itemize}
\item it is concentrated on the family~$\AC^2(I;X)$ of $2$-absolutely continuous curves;
\item it has finite $2$-energy, i.e.\ the right-hand side of~\eqref{eq:WassersteinGeodesic} is finite;
\item it has \emph{bounded compression}, viz.~$(\ev_t)_\pfwd \boldpi\leq C\mssm$ for some contant~$C>0$ independent of~$t\in I$.
\end{itemize}

\begin{defs}[Measured-length space,~{\cite[Dfn.~3.16]{GigHan18}}]\label{d:GigHan}
A metric measure space $(X,\mssd, \mssm)$ is a \emph{measured-length space} if there exists an $\mssm$-co-negligible subset~$\Omega\subset X$ with the following property. For~$i=0,1$ and every~$x_i\in \Omega$ there exists~$\eps\eqdef\eps(x_0,x_1)>0$ such that for each~$\eps_i\in (0,\eps]$ there exists a test plan~$\boldpi^{\eps_0,\eps_1}\in\msP(\mcC(I, X))$ with the following properties:
\begin{enumerate}[$(a)$]
\item\label{i:d:GigHan:1} the map
\begin{equation}\label{eq:d:GigHan:1}
(0,\eps]^\tym{2}\ni(\eps_0,\eps_1)\longmapsto \boldpi^{\eps_0,\eps_1}
\end{equation}
is weakly Borel measurable, viz.
\begin{align*}
(\eps_0,\eps_1)\longmapsto \int \phi \diff\boldpi^{\eps_0,\eps_1} ~\textrm{is Borel measurable}~\qquad \phi\in \Cb(\mcC(I,X)) \semicolon
\end{align*}

\item\label{i:d:GigHan:2} letting~$\ev_t\colon \mcC(I,M)\ni\gamma\mapsto \gamma_t\in M$ be the evaluation map on curves, it holds that
\begin{align*}
(\ev_i)_\pfwd \boldpi^{\eps_0,\eps_1}=\frac{\car_{B_{\eps_i}(x_i)}}{\mssm B_{\eps_i(x_i)}} \comma \qquad \eps_i\in (0,\eps] \semicolon
\end{align*}

\item\label{i:d:GigHan:3} we have that
\begin{align*}
\limsup_{\eps_0,\eps_1\downarrow 0} \iint_I \abs{\dot\gamma}^2 \diff t \diff \boldpi^{\eps_0,\eps_1}(\gamma) \leq \mssd(x_0,x_1)^2 \fstop
\end{align*}
\end{enumerate}
\end{defs}

\begin{prop}[{\cite[Prop.~3.19]{GigHan18}}]\label{p:GigHan}
Every a.e.-locally doubling measured-length \linebreak space satisfies~\ref{eq:SL}.
\end{prop}

\begin{thm}\label{t:Main}
Every $\QCD$ space is a measured-length space.
\begin{proof}
Let~$x_i\in X$,~$i=0,1$, and assume~$x_0\neq x_1$.

\paragraph{Definition of~$\boldpi$} Set~$\eps\eqdef \mssd(x_0,x_1)/4>0$ and, for~$\eps_i<\eps$, let
\begin{align}\label{eq:ApproximatingMeasures}
\mu^{\eps_i}\eqdef \frac{\car_{B_{\eps_i}(x_i)}}{\mssm B_{\eps_i}(x_i)}\cdot\mssm \comma
\end{align}
and
\begin{align*}
G^{\eps_0,\eps_1}\eqdef \set{\gamma\in\Geo(X,\mssd), \gamma_i\in B_{\eps_i}(x_i), i=0,1}\fstop
\end{align*}

By definition of~$\eps_0,\eps_1,\eps$, the sets~$\supp \mu^{\eps_0}$ and~$\supp \mu^{\eps_1}$ are well-separated, and compact since~$(X,\mssd)$ is proper.
Therefore, by Definition~\ref{d:Monge}\iref{i:d:Monge:1}, there exists an optimal dynamical plan~$\boldpi^{\eps_0,\eps_1}\in \OptGeo(\mu^{\eps_0},\mu^{\eps_1})$.

\paragraph{Properties of~$\boldpi$}
Note that~$\boldpi^{\eps_0,\eps_1}$ is concentrated on~$G^{\eps_0,\eps_1}$.
By~\cite[Thm.~5]{Lis07}, every optimal dynamical plan is concentrated on~$\AC^2(I;X)$ and has finite $2$-energy.
Therefore, an optimal dynamical plan is a test plan if and only if it has bounded compression.
Let us show that~$\boldpi^{\eps_0,\eps_1}$ has bounded compression.
For~$t\in I$ let~$\mu^{\eps_0,\eps_1}_t\eqdef (\ev_t)_\pfwd \boldpi^{\eps_0,\eps_1}=\rho_t\mssm$ be the $W_2$-geodesic connecting~$\mu^{\eps_0}$ to~$\mu^{\eps_1}$.
By the quasi curvature-dimension condition,~\eqref{eq:QCD} holds for some~$Q\geq 1$,~$K\in\R$, and~$N\in (1,\infty)$.
As a consequence, for~$\boldpi^{\eps_0,\eps_1}$-a.e.~$\gamma$,
\begin{align*}
\rho_t(\gamma_t)\leq &\ Q \tonde{\tau^\sym{1-t}_{K,N}\ttonde{\mssd(\gamma_0,\gamma_1)}\, \rho_0^{-1/N}(\gamma_0)+ \tau^\sym{t}_{K,N}\ttonde{\mssd(\gamma_0,\gamma_1)}\, \rho_1^{-1/N}(\gamma_1)}^{-N}
\\
=&\ Q \tonde{\tau^\sym{1-t}_{K,N}\ttonde{\mssd(\gamma_0,\gamma_1)}\, \mssm B_{\eps_0}(x_0)^{1/N}+ \tau^\sym{t}_{K,N}\ttonde{\mssd(\gamma_0,\gamma_1)}\, \mssm B_{\eps_1}(x_1)^{1/N} }^{-N}
\\
\leq&\ Q\,\min\set{\mssm B_{\eps_0}(x_0),\mssm B_{\eps_1}(x_1)}^{-1} \tonde{\tau^\sym{1-t}_{K,N}\ttonde{\mssd(\gamma_0,\gamma_1)}+ \tau^\sym{t}_{K,N}\ttonde{\mssd(\gamma_0,\gamma_1)} }^{-N}
\end{align*}
which is finite uniformly in~$t\in I$ since~$\tau^\sym{t}_{K,N}(\theta)+\tau^\sym{1-t}_{K,N}(\theta)$ is bounded away from~$0$ uniformly in~$t\in I$ locally uniformly in~$\theta\in [0,\infty)$.
Since~$\rho_t$ is concentrated on~$\ev_t(G^{\eps_0,\eps_1})$, the previous inequality concludes that~$\boldpi^{\eps_0,\eps_1}$ has bounded compression.

\smallskip

In order to show Definition~\ref{d:GigHan}\iref{i:d:GigHan:3}, note that, since~$\boldpi^{\eps_0,\eps_1}\in \OptGeo(\mu^{\eps_0}, \mu^{\eps_1})$, then, in the constant speed parametrization,
\begin{align*}
\iint_I \abs{\dot\gamma}_t^2 \diff t\diff\boldpi^{\eps_0,\eps_1}(\gamma)\leq \ttonde{\mssd(x_0,x_1)+\eps_0+\eps_1}^2\xrightarrow{\ \eps_0,\eps_1\downarrow 0\ } \mssd(x_0,x_1)^2\comma
\end{align*}
which proves the assertion.

\smallskip

It remains to show the measurability assertion in Definition~\ref{d:GigHan}\iref{i:d:GigHan:1}.
To this end, it suffices to note that, letting~$E_i\colon \eps_i\mapsto \mu^{\eps_i}$,
\begin{align*}
\Pi\circ (E_0, E_1)\colon (\eps_0,\eps_1)\longmapsto \boldpi^{\eps_0,\eps_1}
\end{align*}
with~$\Pi$ as in~\eqref{eq:Pi}.
Since~$(X,\mssd)$ is proper by Corollary~\ref{c:Proper}, choosing~$K_i\eqdef B_\eps(x_i)$ in Corollary~\ref{c:ContPi} shows that the map~$\Pi$ is narrowly/narrowly continuous (hence Borel) on the image of~$(E_0, E_1)$.
Thus, it suffices to show that~$E_i$ is Borel for~$i=0,1$.
We show that~$E_i$, $i=0,1$, is Euclidean/narrowly continuous.
Let~$f\in\mcC_b(X)$ be fixed. It suffices to show that~$r\mapsto E_i(r) f$ is continuous on~$[0,\eps)$.
The continuity at~$r=0$ holds since~$\mssm$ has no atoms.
For~$r,s>0$, we have that
\begin{align*}
\abs{E_i(r)f-E_i(s)f}\leq&\ \abs{\frac{1}{\mssm B_r(x_i)}-\frac{1}{\mssm B_s(x_i)}} \int_{B_r(x_i)} \abs{f}\diff \mssm
\\
&+\ \frac{1}{\mssm B_s(x_i)}
\int_{B_r(x_i)\triangle B_s(x_i)}\abs{f} \diff\mssm \fstop
\end{align*}
The second term vanishes as~$r\to s$ by continuity of the measure~$\mssm$.
As for the first term, it suffices to show that~$r\mapsto \mssm B_r(x_i)$ is continuous for~$r\in [0,\eps)$. This follows from the Portmanteau Theorem, since all balls in~$X$ are continuity sets for~$\mssm$ by the first assertion in Corollary~\ref{c:ContinuitySet}.

\smallskip

If~$x_0=x_1$, then the requirements in Definition~\ref{d:GigHan}\iref{i:d:GigHan:2}--\iref{i:d:GigHan:3} hold trivially, and Definition~\ref{d:GigHan}\iref{i:d:GigHan:1} holds as in the case~$x_0\neq x_1$ discussed above.
\end{proof}
\end{thm}

Combining Corollary~\ref{c:LocalDoubling} with Proposition~\ref{p:GigHan} and Theorem~\ref{t:Main}, we conclude the Sobolev-to-Lipschitz property for $\QCD$ spaces.

\begin{cor} \label{cor: SLQCD}
Every $\QCD$ space satisfies~\ref{eq:SL}.
\end{cor}

\blue{
\begin{rem}
It would not be difficult to show that the original proof of~\ref{eq:SL} for~$\CD(K,N)$ spaces~\cite[p.~48]{Gig13} can be adapted as well to the case of~$\QCD(Q,K,N)$ spaces.
The proof in~\cite{Gig13} relies on an argument in~\cite{Raj12} providing suitable test plans connecting the approximating measures~$\mu^{\eps_i}$ in~\eqref{eq:ApproximatingMeasures}.

The proof we presented above makes instead use of the more refined notion of measured-length space in~\cite{GigHan18}.
Whereas slightly more involved, this proof makes more explicit the relation between the inequality~\eqref{eq:QCD} defining the $\QCD$, and the upper bound on the compression for the aforementioned test plans.
\end{rem}
}

\section{Properties of $\RQCD$ spaces}\label{sec: RQCD}
In this section, we prove several applications of \ref{eq:SL} under the additional assumption that the Cheeger energy is quadratic.

\begin{defs}[cf.~{\cite[\S7.2]{Mil21}}]
Let~$(X,\mssd,\mssm)$ be a metric measure space,~$Q\geq 1$,~$K\in\R$, and~$N\in (1,\infty)$.
We say it satisfies the \emph{Riemannian quasi curvature-dimension condition}~$\RQCD(Q,K,N)$ if it satisfies~$\QCD(Q,K,N)$ and it is additionally infinitesimally Hilbertian, i.e.~$\Ch[\mssd,\mssm]$ is a \emph{quadratic} functional.
\end{defs}

In the following, we omit the indices~$Q$,~$K$, and~$N$ whenever not relevant.
Note that, when~$(X,\mssd,\mssm)$ is an $\RQCD$ space, the quadratic form induced by the Cheeger energy of~$(X,\mssd,\mssm)$ by polarization is a Dirichlet form, again denoted by~$\Ch[\mssd,\mssm]$ and still called the Cheger energy of~$(X,\mssd,\mssm)$.

Recall that a Dirichlet form~$(\mcE,\mcF)$ on a locally compact Polish space~$(X,\mssd,\mssm)$ is called \emph{regular} if~$\mcF\cap \mcC_0(X)$ is both $\ttonde{\mcE(\emparg)+\norm{\emparg}^2_{L^2}}^{1/2}$-dense in~$\mcF$ and uniformly dense in~$\mcC_0(X)$.

\begin{prop}\label{p:Regularity}
Let~$(X,\mssd,\mssm)$ be an $\RQCD$ space. 
Then~$(\Ch[\mssd,\mssm], W^{1,2})$ is a regular strongly local Dirichlet form.
\end{prop}
\begin{proof}
The regularity follows directly from the uniform density of~$\Lip_{bs}(\mssd)$ in~$\mcC_0(X)$ and from the norm density of~$\Lip_{bs}$ in~$W^{1,2}$.
The strong locality is then a standard consequence of the locality of the weak upper gradient~$\slo[w]{\emparg}$.
\end{proof}

\begin{ese}[sub-Riemannian manifolds]\label{ese:SubRiem}
Let~$(M,\mcH)$ be a sub-Riemannian manifold with smooth non-holonomic distribution~$\mcH$ on~$TM$.

On the one hand, by the Chow--Rashevskii Theorem, endowing~$M$ with its Carnot--Carath\'eodory distance~$\mssd_\cc$ and with a smooth measure~$\mssm$ turns it into a proper metric measure space, thus admitting a Cheeger energy~$\Ch[\mssd_\cc,\mssm]$.
On the other hand, the sub-Laplacian~$L$ induced by the distribution~$\mcH$ generates a regular Dirichlet form~$(\mcE,\mcF)$ on~$L^2(\mssm)$ with core~$\mcC^\infty_c(X)$, defined by~$\mcE(f,g)=\scalar{f}{-Lg}_{L^2(\mssm)}$, see e.g.~\cite[p.~191]{BauGar17}.

In fact, for a sub-Riemannian manifold~$(M,\mssd_\cc,\mssm)$ as above, the Dirichlet form $(\mcE,\mcF)$ coincides with the Cheeger energy~$(\Ch[\mssd_\cc,\mssm], W^{1,2})$ of~$(M,\mssd_\cc,\mssm)$, viz.
\begin{align}\label{eq:E=Ch}
(\mcE,\mcF)=(\Ch[\mssd_\cc,\mssm],W^{1,2})\comma
\end{align}
as we show below.

As a consequence, every sub-Riemannian manifold~$(M,\mssd_\cc,\mssm)$ satisfying the~$\QCD$ condition satisfies as well the~$\RQCD$ condition with same parameters.
See~\cite{Mil21} for relevant families of examples of sub-Riemannian manifolds satisfying the~$\QCD$ condition.
\end{ese}
\begin{proof}[Proof of~\eqref{eq:E=Ch}]
On the one hand, by~\cite[Thm.~6.2]{AmbGigSav14}, the square field~$\slo[w]{f}$ of~$f\in W^{1,2}$ coincides with the minimal $2$-weak upper gradient of~$f$.
On the other hand, by~\cite[Thm.~11.7]{HajKos00}, the square field~$\Gamma(f)$ of~$f\in \mcF\cap\mcC(X)$ coincides as well with the minimal $2$-weak upper gradient of~$f$.
As a consequence,~$\mcE(f)=\Ch[\mssd_\cc,\mssm](f)$ for every~$f\in W^{1,2}\cap \mcC(X)$ as well,
and the conclusion follows since~$W^{1,2}\cap \Cz(X)$ is a core for~$(\Ch[\mssd_\cc,\mssm], W^{1,2})$ by regularity of the latter (Prop.~\ref{p:Regularity}) and since~$\mcF\cap \Cz(X)$ is a core for~$(\mcE,\mcF)$ by definition.
\end{proof}

For a \emph{regular} Dirichlet form~$(\mcE,\mcF)$ on~$L^2(\mssm)$, the \emph{local domain}~$\mcF_\loc$ of~$\mcE$ is defined as the space of all functions~$f\in L^0(\mssm)$ so that, for each relatively compact open~$G\subset X$ there exists~$f_G\in\mcF$ with~$f\equiv f_G$ $\mssm$-a.e.\ on~$G$.
When~$(\mcE,\mcF)$ admits square field~$\mcF\ni (f,g)\mapsto \Gamma(f,g)\in L^1(\mssm)$, the quadratic form~$f\mapsto\Gamma(f)\eqdef \Gamma(f,f)$ naturally extends to the local domain~$\mcF_\loc$, e.g.~\cite[\S4.1.i]{Stu94}.

Let us define the following local versions of the Rademacher-type and Sobolev-to-Lipschitz properties:
\begin{gather}
\tag*{$(\mathsf{Rad})_\loc$}\label{eq:RadLoc}
\bLip(\mssd) \subset W^{1,2}_\loc\comma \qquad \slo[w]{f}\leq \Li{f} \as{\mssm}
\\
\nonumber
\\
\tag*{$(\mathsf{SL})_\loc$}\label{eq:SLLoc}
\begin{gathered}
\text{each~$f\in W^{1,2}_\loc\cap L^\infty(\mssm)$ with~$\slo[w]{f}\leq 1$ $\mssm$-a.e.}
\\
\text{has a Lipschitz $\mssm$-modification~$\rep f$ with~$\Li{\rep f} \leq 1$}\fstop
\end{gathered}
\end{gather}

Again by construction,~\ref{eq:RadLoc} holds on every metric measure space.

\begin{prop} \label{p:RSLloc}
Every~$\RQCD$ space satisfies~\ref{eq:SLLoc}.
\begin{proof}
Since we already have \ref{eq:Rad}, by construction of~$\Ch[\mssd,\mssm]$, and~\ref{eq:SL}, by Corollary~\ref{cor: SLQCD}, in order to localize both properties it suffices to show the existence of good Sobolev cut-off functions, similarly to the proof of Theorem~3.9 in~\cite{AmbGigSav15}.

\blue{%
For every~$n\in \N$ and fixed~$x_0\in X$ set~$\theta_n\colon x \mapsto n\wedge \ttonde{2n-\mssd(x,x_0)}_+$.
Since~$\supp\,\theta_n$ is bounded for every~$n\in \N$, we conclude by~\ref{eq:Rad} that~$\theta_n\in W^{1,2}\cap \mcC_0(X)$ and~$\slo[w]{\theta_n}\leq 1$ $\mssm$-a.e.\ for every~$n\in\N$.

Fix now~$f\in W^{1,2}_\loc\cap L^\infty(\mssm)$ with~$\slo[w]{f}\leq 1$ $\mssm$-a.e.. Without loss of generality, up to an additive constant, we may assume that~$f\geq 0$.
By the local property of~$\slo[w]{\emparg}$, we have that
\begin{align*}
\slo[w]{(\theta_n\wedge f)} =& \car_{\set{\theta_n\leq f}} \slo[w]{\theta_n}+\car_{\set{\theta_n>f}} \slo[w]{f} \leq 1 \as{\mssm}
\end{align*}
by assumption on~$f$ and properties of~$\theta_n$, whence~$\theta_n\wedge f_n\in W^{1,2}$. From~\ref{eq:SL} we conclude that~$\theta_n\wedge f$ has a non-relabeled $\mssd$-Lipschitz $\mssm$-representative.
Analogously,~$\theta_{n+1}\wedge f\in W^{1,2}$ and, for all~$n\geq \norm{f}_{L^\infty}$, we have that~$\theta_{n+1}\wedge f\equiv \theta_n\wedge f\equiv f$ $\mssm$-a.e.\ on~$B_n(x_0)$.
Since~$\mssm$ has full topological support, we conclude that the respective $\mssd$-Lipschitz $\mssm$-representatives coincide everywhere on~$B_n(x_0)$.
As a consequence,~$\seq{\theta_n\wedge f}_n$ is a consistent family of $\mssd$-Lipschitz functions coinciding with~$f$ $\mssm$-a.e.\ on~$B_n(x_0)$ for all (sufficiently large)~$n\in \N$.
The conclusion readily follows letting~$n\to\infty$.
}
\end{proof}
\end{prop}

Denote by~$P_t\colon L^2(\mssm)\to L^2(\mssm)$ the strongly continuous contraction semigroup associated to~$(\Ch[\mssd,\mssm], W^{1,2})$, and, for open sets~$A,B\subset X$, set
\begin{align*}
P_t(A,B)\eqdef \int_A P_t \car_B \diff \mssm = \int_B P_t \car_A \diff \mssm \fstop
\end{align*}

For sets~$A_1,A_2\subset X$ define
\begin{align*}
\mssd(A_1,A_2)\eqdef \inf_{x_i\in A_i} \mssd(x_1,x_2) \fstop
\end{align*}
We now give the first application of \ref{eq:SL}. 
\begin{thm}[Integral Varadhan short-time asymptotics]\label{t:Varadhan}
Every~$\RQCD$ space satisfies the integral Varadhan-type short-time asymptotic
\begin{align*}
\lim_{t\downarrow 0} \ttonde{-2t \log P_t(U_1,U_2)}= \mssd(U_1,U_2)^2\comma \qquad U_1,U_2 \quad \text{open} \fstop
\end{align*}

\begin{proof}
Let now~$G_\bullet\eqdef \seq{G_n}_n$ be an increasing exhaustion of~$X$ consisting of relatively compact open sets.
Choosing such~$G_\bullet$ in~\cite[Prop.~2.26]{LzDSSuz20} shows that the broad local space~$\DzLocB{\mssm}$ of bounded functions with bounded $\Ch[\mssd,\mssm]$-energy defined in~\cite[\S2.6.1]{LzDSSuz20} coincides with the local space of bounded functions with bounded $\Ch[\mssd,\mssm]$-energy defined above, viz.
\begin{align}\label{eq:IdentityDzLocB}
\DzLocB{\mssm}=\set{f\in W^{1,2}_\loc\cap L^\infty(\mssm) : \slo[w]{f}^2\leq 1} \fstop
\end{align}

Equation~\eqref{eq:IdentityDzLocB} shows that ---~for the form~$(\Ch[\mssd,\mssm], W^{1,2})$~--- the local Rademacher-type and Sobolev-to-Lipschitz properties defined above respectively coincide with~\cite[$(\Rad{\mssd}{\mu})$ with~$\mu\eqdef \mssm$ in Dfn.~3.1]{LzDSSuz20} and~\cite[$(\SL{\mu}{\mssd})$ with~$\mu\eqdef \mssm$ in Dfn.~4.1]{LzDSSuz20}.
As a consequence, all the results established in~\cite{LzDSSuz20} apply as well to present setting.

For a Borel~$A\subset X$, let~$\hr{\mssm, A}$ denote the maximal function~\cite[Prop.~4.14]{LzDSSuz20}.
By~\cite[Lem.~4.16]{LzDSSuz20} we have
\begin{align*}
\mssd(\emparg, A)\leq \hr{\mssm, A} \as{\mssm} \comma \qquad A \quad \text{Borel} \fstop
\end{align*}
By~\cite[Lem.~4.19, Rmk.~4.19(b)]{LzDSSuz20}
\begin{align*}
\mssd(\emparg, U)\geq \hr{\mssm, U} \as{\mssm} \comma \qquad U \quad \text{open} \fstop
\end{align*}
Combining the above inequalities thus yields
\begin{align*}
\mssd(\emparg, U)= \hr{\mssm, U} \as{\mssm} \comma \qquad U \quad \text{open}\comma
\end{align*}
and the conclusion follows as in the proof of~\cite[Cor.~4.26]{LzDSSuz20}.
\end{proof}
\end{thm}

Denote again by~$P_t\colon L^\infty(\mssm)\to L^\infty(\mssm)$ the extension of~$P_t\colon L^2(\mssm)\to L^2(\mssm)$ to~$L^\infty(\mssm)$.
An $\mssm$-measurable set~$A\subset X$ is called \emph{invariant} if~$\car_A P_t f\equiv P_t (\car_A f)$ for every~$f\in L^\infty(\mssm)$.
We say that the space~$(X,\mssd,\mssm)$ is \emph{irreducible} if every invariant set is either $\mssm$-negligible or $\mssm$-conegligible.

As a corollary of Theorem \ref{t:Varadhan}, we obtain the irreducibility of~$(X,\mssd,\mssm)$.
\begin{cor}[Irreducibility] \label{c:Irr}
Every~$\RQCD$ space is irreducible.
\begin{proof}
By~\cite[Thm.~4.21]{LzDSSuz20} for every Borel~$A\subset X$ with~$\mssm A>0$ there exists a Borel $\mssm$-version~$\widetilde A$ of~$A$ so that the maximal function~$\hr{\mssm,A}$ defined in \cite[Prop.~4.14]{LzDSSuz20} satisfies~$\hr{\mssm, A}= \mssd(\emparg, \widetilde A)$ $\mssm$-a.e.
As a consequence, for every pair of Borel sets~$A_1,A_2$ with~$\mssm A_1,\mssm A_2>0$, we have that
\begin{align*}
\hr{\mssm}(A_1,A_2)\eqdef \mssm\text{-}\essinf_{y\in A_2}\hr{\mssm, A_1}(y)\leq \mssd(\widetilde A_1,\widetilde A_2)<\infty\fstop 
\end{align*} 
\blue{%
By~\cite[Prop.~5.1]{AriHin05},
\begin{align*}
\hr{\mssm}(A_1,A_2)<\infty \iff P_t(A_1,A_2)>0 \text{\; for every\; } t>0\comma
\end{align*}
whence \begin{align}\label{eq:IrredSemigroup}
P_t(A_1, A_2)\eqdef \int_{A_1}P_t\car_{A_2} \diff\mssm>0
\end{align}
for every pair of Borel sets~$A_1,A_2$ with~$\mssm A_1, \mssm A_2>0$.
Now, argue by contradiction that~$(X,\mssd,\mssm)$ is not irreducible.
Then, there exists a Borel set~$A$ with~$\mssm A, \mssm A^\complement>0$ and so that~$P_t\car_A=\car_A P_t$.
Since~\eqref{eq:IrredSemigroup} does not hold for the pair~$A^\complement,A$, we obtain a contradiction, and~$(X,\mssd,\mssm)$ is therefore irreducible.
}
\end{proof}
\end{cor}

\begin{rem}[About irreducibility on~$\RCD(K,\infty)$ spaces]   \label{rem: IRCD}
We stress that, in fact, the same proof of irreducibility presented above for~$\RQCD$ spaces holds as well on all~$\RCD(K,\infty)$ spaces, which seems not explicitly stated in the existing literature.
The result holds even in the case when~$\mssm X=\infty$, since we only rely on the local Rademacher-type and Sobolev-to-Lipschitz properties~\ref{eq:RadLoc} and~\ref{eq:SLLoc}, which can be proved by the same localization argument as in Proposition~\ref{p:RSLloc}.
Note that these arguments hold even if the space~$(X,\mssd,\mssm)$ is not locally compact, in which case the local domain $ W^{1,2}_\loc$ may be replaced with the broad local domain, see, e.g.,  \cite[\S2.4]{LzDSSuz20}.
Importantly, this proof does not rely on any heat-kernel estimate. 
\end{rem}

As the second application, we prove that~$\mssd$ coincides with the intrinsic distance ---~in the sense of e.g.~\cite[Eqn.~(1.3)]{Stu94}~--- associated with the Cheeger energy, viz.
\begin{align*}
\mssd_{\Ch[\mssd,\mssm]}(x,y)\eqdef \sup\set{f(x)-f(y) : f\in W^{1,2}_\loc\cap \mcC(X)\comma \slo[w]{f}^2\leq 1 \as{\mssm}}\fstop
\end{align*} 

\begin{thm}\label{t:Distances}
Let~$(X,\mssd,\mssm)$ be an $\RQCD$ space. Then,~$(\Ch[\mssd,\mssm], W^{1,2})$ is a regular strongly local Dirichlet form on~$L^2(\mssm)$, and 
\begin{align}
\label{eq:IntD:1}
\mssd(x,y)=&\ \mssd_{\Ch[\mssd,\mssm]}(x,y)
\\
\label{eq:IntD:1.5}
=&\ \sup\set{f(x)-f(y) : f\in W^{1,2}\cap \mcC_c(X)\comma \slo[w]{f}^2\leq 1 \as{\mssm}}\fstop
\end{align}

\begin{proof}
On the one hand, by~\ref{eq:RadLoc} (established in Proposition~\ref{p:RSLloc}) and by a straightforward adaptation of~\cite[Lemma 3.6]{LzDSSuz20}, we have that~$\mssd\leq \mssd_{\Ch[\mssm,\mssd]}$.
On the other hand, by~\ref{eq:SLLoc} (also in Proposition~\ref{p:RSLloc}) and by a straightforward adaptation of~\cite[Prop.~4.2]{LzDSSuz20}, we have that~$ \mssd_{\Ch[\mssm,\mssd]}\leq \mssd$.
Combining the two inequalities shows the equality in~\eqref{eq:IntD:1}.

Since~$\Ch[\mssd,\mssm]$ is regular, and irreducible by Corollary~\ref{c:Irr},~$(X,\mssd_{\Ch[\mssm,\mssd]})=(X,\mssd)$ satisfies Assumption~(A) in~\cite{Stu94}, and the equality in~\eqref{eq:IntD:1.5} follows by~\cite[Prop.~1.c, p.~193]{Stu94}.
\end{proof}
\end{thm}

\begin{rem}
Combining Theorem~\ref{t:Distances} with~\cite[Prop.~2.31]{LzDSSuz20} shows that several definitions for the intrinsic distance~$\mssd_{\Ch[\mssm,\mssd]}$ ---~and in particular the one of~$\mssd_\mssm$ in~\cite[Dfn.~2.28]{LzDSSuz20}~--- coincide.
\end{rem}

In the next corollary, we denote by $p_t$ the density w.r.t.~$\mssm$ of the heat-kernel measure of the heat semigroup~$P_t$.

\begin{cor}[Pointwise Varadhan short-time asymptotics]\label{c:pVaradhan}
Every compact $\RQCD_\reg$ space satisfies the pointwise Varadhan short-time asymptotics
\begin{align}\label{eq:pVaradhan}
\lim_{t\downarrow 0} \ttonde{-2t \log p_t(x,y)}= \mssd(x,y)^2 \fstop
\end{align}
\begin{proof}
As a consequence of the~$\MCP$ condition implicit in the notation for~$\RQCD_\reg$, we have the validity of the local weak $2$-Poincar\'e inequality, see~\cite[Cor.~6.6(ii)]{Stu06b}, and of the local doubling property, as noted in Corollary~\ref{c:LocalDoubling}.
Both properties are in fact global, by compactness of~$X$.
By Proposition~\ref{p:Regularity} and Theorem~\ref{t:Distances}, the Dirichlet form~$(\Ch[\mssd,\mssm],W^{1,2})$ is \emph{strongly regular}, i.e.\ regular and so that the intrinsic distance~$\mssd_{\Ch[\mssd,\mssm]}$ induces the original topology.
By~\cite[Thm.~4.1]{Ram01} we conclude~\eqref{eq:pVaradhan} with~$\mssd_{\Ch[\mssd,\mssm]}$ in place of~$\mssd$.
The conclusion now follows since~$\mssd=\mssd_{\Ch[\mssd,\mssm]}$ by Theorem~\ref{t:Distances}.
\end{proof}
\end{cor}

Finally let us briefly discuss the Lipschitz regularization property of the semigroup $\seq{P_t}_{t\geq 0}$. Let ${\sf c}: [0,+\infty) \to (0, +\infty)$ be a measurable function so that locally uniformly bounded away from~$0$ and infinity.
Following \cite[Definition 3.4]{Ste20}, we say that an $\RQCD$ space satisfies~$\BE_w(\mathsf{c}, \infty)$ if for all $f \in W^{1,2}$ and $t \ge 0$,
\begin{equation}\tag{$\BE_w$}\label{eq:wBE}
\slo[w]{P_tf}^2\le {\sf c}(t)^2\, \ttonde{P_t\slo[w]{f}}^2 \fstop
\end{equation}

\begin{cor}[$L^\infty$-to-$\bLip$-Feller] \label{cor: LL}
Assume that $(X,\mssd,\mssm)$ is an $\RQCD$ space satisfying~\eqref{eq:wBE}. Then, $P_t: L^\infty(X, \mssm) \to \bLip(X, \mssd)$ for any $t>0$ and 
\[
\sqrt{2I_{-2}(t)}\, \Li[\mssd]{P_tf} \le \norm{f}_{L^\infty} \comma \qquad I_{-2}(t)\eqdef\int_0^t\mssc^{-2}(s)\diff s \fstop
\]
\begin{proof}
By Theorem~\ref{t:Main}, the space satisfies~\ref{eq:SL}, i.e.~\cite[(P.5)]{Ste20}, and the assertion holds by~\cite[Cor.~3.21]{Ste20}.
\end{proof}
\end{cor}

{\small

}

\end{document}